\documentclass{amsart}
\usepackage{hyperref}
\usepackage{setspace}
\usepackage{amsmath,amssymb,amsthm,mathabx}
\usepackage[margin=3cm]{geometry}
\usepackage{fancyhdr}
\usepackage{enumitem}
\usepackage{cleveref}
\usepackage{tikz}

\usepackage{color}

\newcommand{\obp}{\otimes^{\gamma}}
\newcommand{\N}{\mathbb N}

\newcommand{\ot}{\otimes}
\newcommand{\mU}{\mathcal U}
\newcommand{\mT}{\mathcal T}
\newcommand{\mB}{\mathcal B}

\newcommand{\C}{\mathbb C}
\newcommand{\mA}{\mathcal{A}}

\theoremstyle{definition}

\allowdisplaybreaks

\theoremstyle{cupthm}
\newtheorem{theorem}{Theorem}[section]
\newtheorem{prop}[theorem]{Proposition}
\newtheorem{corollary}[theorem]{Corollary}
\newtheorem{lemma}[theorem]{Lemma}
\newtheorem{remark}[theorem]{Remark}
\newtheorem{conjecture}[theorem]{Conjecture}
\theoremstyle{cupdefn}

\numberwithin{equation}{section}

\title[Second dual of generalized group algebras]{Annihilators in the bidual of the generalized group algebra of a discrete group.}

\author{Lav Kumar Singh}
\address{School of Physical Sciences, Jawaharlal Nehru University, New
  Delhi }

\email{lavksingh@hotmail.com} 


\thanks{This research work was carried out with the financial support
	from the Council of Scientific and Industrial Research (Government
	of India) through a Senior Research Fellowship with
	No. \bf 09/263(1133)/2017-EMR-I}
\begin{document}
\maketitle
\begin{abstract}
In this short note, the second dual of generalized group algebra $(\ell^1(G,\mA),\ast)$ equipped with both Arens product is investigated, where $G$ is any discrete group and $\mA$ is a Banach algebra containing a complemented algebraic copy of $(\ell^1(\mathbb N),\bullet)$. We give an explicit family of annihilators(w.r.t both the Arens product) in the algebra $\ell^1(G,\mA)^{**}$, arising from non-principal ultrafilters on $\mathbb N$ and which are not in the toplogical center. As a consequence, we also deduce the fact that $\ell^1(G,\mA)$ is not Strongly Arens irregular.
\end{abstract}

\section{Introduction}
 For any Banach algebra $\mA$, Richard Arens (in \cite{Arens}) defined
 two products $\Box$ and $\Diamond$ on its bidual  space $\mA^{**}$ such that each product
 makes $\mA^{**}$ into a Banach algebra and the canonical isometric
 inclusion $J:\mA\to \mA^{**}$ becomes a homomorphism with respect to both
 the products. A Banach algebra $\mA$ is said to be \textit{Arens
   regular} if the two products $\Box$ and $\diamond$ agree on
 $\mA^{**}$, i.e. if $f\Box g=f\diamond g$ for all $f,g\in \mA^{**}$. Even
 when a Banach algebra is not Arens regular, the size of its
 topological center carry a great deal of information.  Dales and Lau
 introduced the notion of strong Arens irregularity to distinguish
 between non-Arens regular Banach algebras. 
Briefly speaking, a Banach algebra $\mA$ is said to be {\em left}
(resp., {\em right}) {\em strongly Arens irregular} if its so called
{\em left topological center} $Z_t^{(l)} (\mA^{**})$ (resp., {\em
  right topological center} $Z_t^{(r)} (\mA^{**})$) equals $\mA$ (see \ref{def1} for definitions). And 
$\mA$ is said to be {\em strongly Arens irregular} (in short, S.A.I.)
if it is both left and right strongly Arens irregular.  Interestingly,
on the other extreme, $\mA$ is Arens regular if and
only if $Z_t^{(l)} (\mA^{**}) = \mA^{**} = Z_t^{(r)}
(\mA^{**})$. Reflexive spaces turn out to be at the junction of Arens
regularity and strong Arens irregularity because
$\mA=\mA^{**}=Z^{(l)}(\mA^{**})=Z^{(t)}(\mA^{**})$.\smallskip

 After the pioneering paper \cite{Arens} by Richard Arens, people
 started investigating the Arens regularity of various Banach algebras
 that appear in Abstract Harmonic Analysis. The group algebra $L^1(G)$
 has been a center of attention for quite some time. Arens himself
 proved that the semi-group algebra $\ell^1(\mathbb N)$ (with respect
 to convolution) is not Arens regular. N.J Young in \cite{Young} showed that the group algebra $L^1(G)$ is never Arens regular for any infinite locally compact group $G$.  Graham, in \cite{Graham},
 gave an easy proof for the irregularity of $L^1(G)$ for the case when $G$ is infinite discrete or amenable locally compact group.
  Losert and Lau in \cite{Lau-losert}
 proved that $L^1(G)$ is strongly Arens irregular for all infinite
 locally compact groups.\smallskip

In a more general setup, one can consider the space $L^1(G,\mA)$ of all $\mA$ valued Bochner integrable functions, where $G$ is any locally compact group and $\mA$ is a Banach algebra. The space $L^1(G,\mA)$ becomes a Banach algebra with respect to convolution multiplication and $||\cdot||_1$ norm (see \ref{generalised group algebras} for details).  It is then quite natural to analyze the strong Arens irregularity of
 generalized group algebras $L^1(G,\mA)$ for infinite locally compact
 groups $G$ and Banach algebras $\mathcal A$. This short article is just a
 report on some progress made in this direction. A lot still remains
 to be explored. Since $L^1(G)$ is not Arens regular, and because
 $L^1(G, \mA) \cong L^1(G) \obp \mA$, we know from \cite{Ulger} that
 $L^1(G,\mA)$ is also not Arens regular. Question remains under what
 conditions on $\mA$, the Banach algebra $L^1(G,\mA)$ is strongly
 Arens irregular. Needless to mention, if $\mA=\mathbb C^n$ with
 point-wise multiplication and $G$ is infinite, then
 $L^1(G, \mA)$ is strongly Arens irregular (due to
 \cite{Lau-losert}). Further, since there is an isometric algebra
 isomorphism from $L^1(G)\obp L^1(H)$ onto $L^1(G\times H)$
 (w.r.t. convolution on both the legs) - see \cite{Kaniuth}, it follows
 again from \cite{Lau-losert} that $L^1(G,\mA)$ is strongly Arens
 irregular for $\mA=L^1(H)$, where $H$ is an infinite group. However,
 $L^1(G,\ell^1):=\ell^1(L^1(G))$ is not strongly Arens  irregular when $\ell^1$ is equipped with pointwise product. We shall prove this in the next section.
 This is interestingly in contrast to the above mentioned fact that
 the (infinite) group algebra $L^1(G)$ is always S.A.I. - see
 \cite{Lau-losert}. 
 We study the bidual of $L^1(G,\ell^1)$ for discrete groups  $G$ and $\ell^1$ equipped with pointwise product, and show that it is rich in annihilators. We give a family of functionals, generated by non-principal ultrafilters, in the bidual of $L^1(G,\ell^1)$, which are annihilators of the bidual with respect to both the Arens product and are not an element of $L^1(G,\ell^1)$. As a consequence, we also deduce the failure of SAI in generalized group algebras explicitly (which is well known as mentioned in previous paragraph). Getting hold of functionals in bidual of a Banach algebra which are not in the algebra itself is usually a difficult task. The ultrafilter technique demonstrated in third section seems to be a powerful tool to study biduals of some Banach algebras.

\section{Preliminaries}

\subsection{Arens regularity and strong Arens irregularity}\label{def1}
We quickly recall the definitions of the two products $\Box$ and
$\Diamond$.  Let $\mA$ be a Banach algebra. For $a \in \mA$,
$\omega\in \mA^*$, $f\in \mA^{**}$, consider the functionals
$\omega_a, {}_{a}\omega\in \mA^*$, $\omega_f,{}_f\omega\in \mA^{**}$
given by
\[
w_a=(L_a)^*(\omega),
{}_a\omega=(R_a)^*(\omega); \omega_f(a)=f({}_a\omega)\text{ and }
{}_f\omega(a)=f(\omega_a).
\]
Then, for $f,g\in \mA^{**}$ the
operations $\Box$ and $\Diamond$ are given by
\[
(f\Box g)
(\omega)= f({}_g\omega)\text{ and } (f\Diamond
g)(\omega)=g(\omega_f)
\]
for all $\omega \in \mA^*$. And, $\mA$ is said to be {\em Arens regular} if the
products $\Box$ and $ \Diamond$ agree.

Further, the left and the right topological centers of $\mA$ are defined,
respectively, as
\[
\begin{array}{rcl}
Z_t^{(l)} (\mA^{**}) & = & \{\varphi \in \mA^{**}: \varphi\Box
\psi=\varphi\Diamond\psi \text{ for all } \psi\in \mA^{**} \}; \text{ and }
\\ Z_t^{(r)} (\mA^{**}) & = & \{\varphi \in \mA^{**}:
\psi\Box\varphi=\psi\Diamond \varphi \text{ for all } \psi\in \mA^{**} \}.
\end{array}
\]
Identifying $\mA$ with $J(\mA)$, where $J : \mA \to \mA^{**}$ is the
canonical isometry, it is known that $\mA \subseteq Z_t^{(l)}
(\mA^{**}) \cap Z_t^{(r)} (\mA^{**})$.  It is easily seen that $\mA$ is Arens
regular if and only if $Z_t^{(l)} (\mA^{**}) = \mA^{**} = Z_t^{(r)}
(\mA^{**})$.   A Banach algebra $\mA$ is said
to be {\em left} (resp., {\em right}) {\em strongly Arens irregular}
if $Z_t^{(l)} (\mA^{**}) = \mA$ (resp., $Z_t^{(r)} (\mA^{**}) =
\mA$). And, $\mA$ is said to be {\em strongly Arens irregular} if it
is both left and right strongly Arens irregular.
\subsection{Projective tensor product}
Let $X$ and $Y$ be Banach spaces. Then,  there
are various ways of imposing a normed space  structure on their algebraic tensor product $X
\otimes Y$.   For instance, for each $u \in X \otimes Y$, its
projective norm is given by
\[
\Vert u \Vert_\gamma =\inf\left\{\sum_{i=1}^{n}\Vert x_i\Vert
\Vert y_i\Vert~:~u=\sum_{i=1}^{n}x_i\otimes y_i\right\}.
\]
This norm turns out to be a \textit{cross norm}, i.e., $\|x
\otimes y\|_{\gamma} = \|x\| \|y\|$ for all $(x,y) \in X \times
Y$; and, the completion of the normed space $X \otimes Y$ with
respect to this norm is  usually denoted by
$X \otimes^\gamma Y$.

Let $\mA$ and $\mB$ be Banach algebras. Then, there is a natural
multiplication structure on their algebraic tensor product $\mA
\otimes \mB$ satisfying $(a_1\otimes b_1)(a_2\otimes
b_2)=a_1a_2\otimes b_1b_2$ for all $a_i \in \mA$ and $b_i \in \mB$,
$i=1,2$. And, there are various ways of imposing a normed algebra
structure on $\mA \otimes \mB$.  For instance, completing $\mA \otimes
\mB$ with respect to the projective norm gives a Banach algebra
structure on $\mA \obp \mB$, which occupies a prominent space  in the
theory of tensor products of Banach algebras.

\subsubsection{Generalized group algebras} \label{generalised group algebras}
Let $G$ be a locally compact group, $\mA$ be a Banach algebra and $\mu$ denote the Haar measure on $G$. Then,
recall that, $L^1(G,\mA)$ denotes the space of all Bochner integrable
functions from $G$ to $\mA$ with the norm $||f||_1:=\int_G||f||d\mu$
and forms a Banach algebra with respect to the convolution product
defined as $f\star g(x)=\int_Gf(s)g(s^{-1}x)ds$ - see \cite{Kaniuth}
for further details on generalized group algebras. When $G$ is a
discrete group, then $L^1(G,\mA)$ is nothing but the collection of all
absolutely summable $\mA$-valued functions and $||f||_1=\sum_{g\in
  G}||f(g)||$ while the convolution reduces to $f\ast g(x)=\sum_{s\in G}f(s)g(s^{-1}x)$ (notice that this is a sum of the elements of $\mA$). We will use the notion $\ell^1(G,\mA)$ for the $\mA$-valued group algebra when $G$ is discrete.

The following natural identifications are well known  and will be used ahead - see, for
instance, \cite{Kaniuth, Ryan}.
\begin{theorem}\label{obp-facts}
  \begin{enumerate}
  \item For any two Banach spaces $X$ and $Y$, there is a canonical
    surjective isometry $\theta: B(X, Y^*) \to (X \obp Y)^* $
    satisfying
    \[
\theta(f)(x\obp y) = f(x)(y) \text{ for all } f \in B(X, Y^*), x \in X, y \in Y.
    \]
    ( $B(X,Y^*)$ denotes the space of all bounded linear maps from $X$ to $Y^*$).
    
\item For any locally compact group $G$ and Banach algebra $\mA$, there
  exists an isometric algebra isomorphism $T:L^1(G)\obp \mA\to L^1(G,\mA)$
  satisfying
  \[
  T(f\otimes a)(g)=f(g)a \text{ for all } f \in f \in L^1(G), a \in \mA, g \in G.
  \]
  \end{enumerate}
\end{theorem}
\noindent The following elementary observations which are easy to prove will be used
ahead. 
\begin{lemma}\label{countable-support}
	Let $G$ be an infinite discrete group and $\mA$ be a Banach
        algebra. Then,
        \begin{enumerate}
\item if $x= \sum_{n\in F} \delta_{g_n} \ot a_n \in L^1(G) \obp \mA$
  for some $F \subseteq \N$, a  collection $\{g_n: n \in F\} \subseteq G$ with no repetitions and a collection $\{a_n: n
  \in F\} \subset \mA$,  then $\|x\|_\gamma= \sum_{n\in F}\|a_n\|$; and,
\item for each $x\in L^1(G) \obp \mA$, there exists
        an $F \subset \N$, a  collection $\{h_n: n \in F\}\subset G$ with no repetitions and
an absolutely summable collection        $\{b_n: n \in F\}$ in $ \mA\setminus \{0\}$ such that $x = \sum_{n\in F}
        \delta_{h_n} \ot b_n $. 
\end{enumerate}
\end{lemma}
\section{Annihilators in Bidual of $\ell^1(G,\mA)$}
Throughout this section, $G$ will denote an infinite discrete
group. We will give a proof of the fact that the convolution algebra $\ell^1(G,\ell^1(\mathbb N))$ is not S.A.I (where $\ell^1(\mathbb N)$ is equipped with pointwise multiplication). 
\begin{theorem}
The Banach algebra $(\ell^1(G,\ell^1(\mathbb N)),\ast)$ is not strongly Arens irregular. 
\end{theorem}
\begin{proof}
	Since, $\left(\ell^1(\mathbb Z,\ell^1(\mathbb N),\ast\right):=\left(\ell^1(\mathbb Z),\ast\right)\obp\left(\ell^1(\mathbb N),\bullet\right)=\left(\ell^1(\mathbb N,(\ell^1(\mathbb Z),\ast)),\bullet\right)$ and $\ell^1(\mathbb Z)$ is Arens irregular, we know by \cite{Ulger}(Cor. 3.5 ) that $\left(\ell^1(\mathbb Z,\ell^1(\mathbb N)),\ast\right)$ is not Arens regular. Let $X$ be the Banach algebra $X=(\ell^1(G),\ast)$. Then  by above identifications, our  generalized group algebra $\left(\ell^1(G,\ell^1(\mathbb N)),\ast\right)$ is isometrically isomorphic to the Banach algebra $(\ell^1(\mathbb N,X),\bullet)$.\\ Consider the trivial embedding $i:c_0(\mathbb N,X^*)\to \ell^\infty(\mathbb N, X^*)=\ell^1(\mathbb N, X)^*$ and a natural map $$K:\ell^1(\mathbb N, X^{**})\to \ell^1(\mathbb N,X)^{**}:=\ell^\infty(\mathbb N,X^*)^*$$ where $K(f)(g)=\sum_{n\in \mathbb N}\left<f(n),g(n)\right>$ for each $f\in \ell^1(\mathbb N,X^{**})$ and $g\in \ell^\infty(\mathbb N,X^*)$. Clearly $K$ is an isometry(not necessarily surjective). Let $P=K\circ i^*:\ell^1(\mathbb N, X)^{**}\to \ell^1(\mathbb N,X)^{**}$. One can easily verify that the map $P$ is a projection i.e $P^2=P$ and the range of $P$ is $\ell^1(\mathbb N,X^{**})$. Thus, we have the direct sum decomposition of Banach spaces $\ell^1(\mathbb N, X)^{**}=\text{Range}(P)\oplus \text{Range}(I-P)$. But $\text{Range}(I-P)=i\left(c_0(\mathbb N, X^{*})\right)^\perp$. Hence, we have the following Banach space decomposition \begin{eqnarray}\label{decompose}
		\ell^1(\mathbb N, X)^{**}=\ell^1(\mathbb N, X^{**})\oplus i(c_0(\mathbb N, X^{*}))^\perp
	\end{eqnarray}
	
	Now let $\theta:\ell^\infty(\mathbb N,X^*)\to \ell^1(\mathbb N,X)^{*}$ be  the dual identification given by $\theta(f)(g)=\sum_{n\in\mathbb N}\left<f(n),g(n)\right>$. For any $f\in \ell^\infty(\mathbb N, X^*)$ and $h,g\in \ell^1(\mathbb N, X)$, we have \begin{align*}\left(\theta(f)\right)_h(g)&={_h}\left(\theta(f)\right)(g)\\&=\theta(f)(h\bullet g)\\&=\sum_{n\in \mathbb N}\left<f(n),h(n)\ast g(n)\right>\\&=\sum_{n\in \mathbb N}\left<f(n)\ast\widecheck{h(n)},g(n)\right>\\&=\theta(f\bullet \hat{h})(g)
	\end{align*}
	where $\widecheck{h(n)}(m)=h(n)(-m)$ and $\hat{h}\in \ell^1(\mathbb N,X)$ is defined as $\hat{h}(n)=\widecheck{h(n)}$. It can be easily seen that $f\bullet \hat h\in c_0(\mathbb N, X^*)$. Thus $\omega_h={_h}w\in \theta(c_0(\mathbb N, X^*))$ for each $h\in \ell^1(\mathbb N,X)$ and $w\in \ell^1(\mathbb N,X)^*$. Now if $\nu\in \ell^1(\mathbb N, X)^{**}$ and $\eta\in c_0(\mathbb N, X^{***})^\perp\subset \ell^\infty(\mathbb N,X^{***})^* $, then for the restriction $\eta_o$ of $\eta$ to $\ell^\infty(\mathbb N,X^*)$ we have 
	\begin{align*}
		\nu\square\eta_o(\omega)&=\nu(\omega_{\eta_o})\\&=0&& \left(\text{because}~ \omega_{\eta_o}(h)=\eta_o(\omega_h)=0~\text{as}~\omega_h\in c_0(\mathbb N, X^*)\right)
	\end{align*}
	Similarly, 
	\begin{align*}
		\nu\diamond \eta_o(\omega)&=\eta_o(\omega_{\nu_1})+\eta_o(\omega_{{\nu}_2})&& \left(\nu=\nu_1+\nu_2 ~\text{using the decomposition }~\ref{decompose}\right)\\&=\eta_o(\omega_{v_1})+0&&\left(\because\omega_{\nu_2}=0 \right)\\&=0&&\left(\because \omega_{\nu_1}\in c_o(\mathbb N, X^{***})\right)
	\end{align*}
	Thus, we see that $\eta_o$ is right annihilator with respect to both the Arens product on $\ell^1(\mathbb N, X)^{**}$. Further, $\eta_o$  does not belong to $\ell^1(\mathbb N, X)$. Hence, the right topological center of generalized group algebra $\left(\ell^1(G,\ell^1(\mathbb N)),\ast\right)^{**}:=\left(\ell^1(\mathbb N, X),\bullet\right)^{**}$ is non-trivial.
\end{proof}

For any general Banach algebra $\mA$, it is not always possible to take the above given approach for the topological center of $(\ell^1(G,\mA),\ast)$. The direct sum decomposition as in \ref{decompose} cannot be given for $\ell^1(G,\mA)^{**}$ for any arbitrary Banach algebra $\mA$. We propose a direct access to annihilators in the bidual of $\ell^1(G,\mA)^{**}$ with respect to both the Arens product, using non-principle ultrafilters in the case when $\mA$ contains a complemented algebraic copy of $\ell^1(\mathbb N)$. To do so, we first observe the following crucial result regarding Banach algebras which contains an isomorphic copy of $(\ell^1(\mathbb N),\bullet)$. We say that a closed subspace of $\mathcal B$ of a Banach space $\mathcal A$ is complemented in $\mathcal A$ if $\mathcal A=\mathcal B\oplus \mathcal C$ for some closed subspace $\mathcal C$ of $\mathcal A$, where $||(b,c)||=||b||+||c||$.
\begin{prop}\label{phi-U} 
	Let $\mA$ be a Banach algebra such that there exists an algebra isomorphism $\rho:(\ell^1(\mathbb N),\bullet)\to \mA$ onto the its image such that $\rho(\ell^1(\mathbb N))$ is complemented in $\mA$ and let  $X=\ell^1(G,\mA)$. Then, associated to each non-principal ultrafilter
	$\mU$ on $\N$, there exists a functional $\varphi_{\mU}$ in $X^{**}$ such that  $\varphi_{\mU} \notin 
	J(X)$. This functional is precisely defined as $$\varphi_\mU(\tau)=\lim_{\mU}\tau (\delta_{h_n}\otimes \rho(e_n))   \text{ for all }
	\tau \in X^*.$$
	where, $\{h_n\}_{n\in \mathbb N}$ is a countable set of elements in $G$.
\end{prop}
\begin{proof}
	Let  $\mU$ be a non-principal ultrafilter on $
	\mathbb N$.  Fix a countable set $S=\{h_n\}_{n=1}^\infty$ consisting of  distinct elements of $G$. 
	
	Notice that, for each $\tau \in X^*$, via the identifications illustrated in \Cref{obp-facts}, we have
	\[
	|\tau
	(\delta_{h_n}\otimes \rho(e_n))| \leq \|\tau\|\|\rho\| \text{ for all } n \in \N
	.
	\]
	Thus, since limit along an ultrafilter in a compact set always
        exists, we have a map $\varphi_\mU: X^* \to \C$ given by
	\begin{equation}\label{phi-U-defn}
	\varphi_\mU(\tau)=\lim_{\mU}\tau (\delta_{h_n}\otimes \rho(e_n))   \text{ for all }
	\tau \in X^*.
	\end{equation}
	Clearly, $\varphi_\mU$ is a bounded linear functional on
        $X^*$.
	We first assert that $\varphi_\mU \neq 0$. Consider $\eta:\ell^1(G)\to \mA^*$ defined
	as
	\[
	\eta(\sum_{g\in
		G}\alpha_g\delta_g)=\sum_{i=1}^\infty\alpha_{h_i}E_i^*.
	\]
	where, $E_i^*\in \mA^*$ are some fixed Hahn-Banach extensions of point evaluations $E_i^*(\rho(\{x_j\}_{j=1}^\infty))=x_i$ on the subpsace $\rho(\ell^1(\mathbb N))$ of $\mA$.
	One can
	easily verify $\eta \in B(\ell^1(G), \mA^*)$ and that $||\eta||\leq 1
	$. Clearly, $\varphi_{\mU}(\eta)=1$. Hence, $\varphi_{\mU}\neq 0$.
	\medskip

	Next, consider the subset $\Delta:=\operatorname{span}\{\mT_{g,E^*}~\vert ~g\in
	G, E^*\in \mA^*_1\}$ of $X^*$, where $\mT_{g,E^*}: \ell^1(G) \to \mA^*$ is
	defined as
	\[
	\mT_{g,E^*}(\sum_{h\in G}c_h\delta_h)=c_gE^*.
	\]
	(One can easily see that $\mT_{g,E^*}\in B(\ell^1(G), \mA^*)$ and are uniformly bounded and, hence, in
	view of \Cref{obp-facts}, $\Delta$ is a well defined non-trivial family
	in $X^*$.)
	
	We now claim that $J(x)$ does not vanish on $\overline{\Delta}$ for any
	$0 \neq x\in X$, while $\varphi_{\mU}$ vanishes (on $\Delta$ and
	hence) on $\overline{\Delta}$. Since $\varphi_\mU \neq 0$, this will then imply that $\varphi_\mU
	\notin J(X)$.
	
	The latter claim is evident from the fact that
	$\varphi_{\mU}(\mT_{g,i})=0$ for each $g\in G$ and $i\in \mathbb
	N$.
	
	To prove the first part of the claim, suppose $x \in
        X\setminus \{0\}$. Then,  as observed in \Cref{countable-support},
        there exist countable sets $F\subset \N$, $\{u_j : j\in F\}
        \subset G$ and $\{b_j: j \in F\}\subset \mA\setminus\{0\}$ such that
        $x=\sum_{j=1}^\infty\delta_{u_j}\otimes b_j$. 
For some fixed $i$, choose $E^*$ such that $E^*(b_i)\neq 0$.
        Now, notice that \begin{align*}
         J(x)(\mT_{u_i,E^*})&=\mT_{u_i,E^*}(x)\\&=\sum_{j=1}^\infty\mT_{u_i,E^*}(\delta_{u_j})(b_j)\\&=E^*(b_{i})\neq 0
	\end{align*}
	Thus, $\varphi_\mU\in X^{**}\setminus J(X)$.
\end{proof}

\begin{theorem} \label{mainth}
	Let $G$ be an infinite discrete group, $\mA$ be a Banach
        such that $\rho:(\ell^1(\mathbb N),\bullet)\to \mA$ is an isomorphism onto its image such that the image is complemented in $\mA$ and $\mU$ be a
        non-principal ultrafilter on $ \mathbb N$.  Then,
        $\varphi_{\mathcal U}$ is left annihilator of $\ell^1(G,\mathcal A)^{**}$ for both the Arens products.

\end{theorem}
\begin{proof}
  Let $ X:=\ell^1(G,\mA)$. Cosider $S :=\{h_n\} \subseteq G$ as in
  \Cref{phi-U}. Then, for any $\psi\in X^{**}$ and $\mT\in X^*$, we
  have
  \begin{align*}
  ( \varphi_\mU\square \psi)(\mT)&=\lim_{\mU}\ {_\psi}\mT(\delta_{h_n}\otimes \rho(e_n))\\
  &=\lim_{\mathcal U}\psi(\mathcal T_{\delta_{h_n}\otimes \rho(e_n)}).
  \end{align*}
	We claim that $\lim_{n}\psi(\mT_{\delta_{h_n}\otimes
          \rho (e_n)})=0$.

        To prove this we assume the contrary. Then, there exists an
        $\epsilon >0$ and an increasing sequence $\{n_t\}_{t\in
          \mathbb N}\subset \N$ such that
        $|\psi(\mT_{\delta_{h_{n_t}}\otimes \rho(e_{n_t})})|\geq \epsilon $
        for all $t\in \mathbb N$. Consider the functionals
        $h_N:=\sum_{t=1}^Nc_{n_t}\mT_{\delta_{h_{n_t}}\otimes
          \rho(e_{n_t})}$ for $N \in \N$, where
        $c_{n_t}:=\frac{\overline{\psi(\mT_{\delta_{h_{n_t}}\otimes
              \rho(e_{n_t})})}}{|\psi(\mT_{\delta_{h_{n_t}}\otimes
            \rho(e_{n_t})})|}$. Let  $x:=\sum_{i=1}^k\lambda_i\delta_{g_i}\otimes a_{s_i}\in
        \ell^1(G)\otimes \mA$ be an arbitrary element, where  $\{g_i : 1
        \leq i \leq k\}\subset G$ and $\{s_i : 1 \leq i \leq
        k\}\subset\N$ (with possible repetitions) are two infinite sets with possible repitions. Further, let $a_{s_i}=\rho(\sum_{j=1}^\infty x^{(s_i)}_je_j)+v_{s_i}$ (since image of $\rho$ is complemented in $\mA$).
       	\begin{align*}
	|h_N(x)|&=\left|\mT(\sum_{t=1}^N\sum_{i=1}^k\lambda_i c_{n_t}\delta_{h_{n_t}}\star\delta_{g_i}\otimes \rho(e_{n_t})\cdot a_{s_i})\right|\\
	&=\left|\mT(\sum_{t=1}^N\lambda_{n_t}c_{n_t}\delta_{h_{n_t}g_{n_t}}\otimes x_{n_t}^{(s_i)})\right|& & \textrm{($\because$~ $\rho(e_{n_t})\cdot a_{s_i}=x_{n_t}^{(s_i)}$ )}\\
&\leq||\mT||\sum_{t=1}^N|\lambda_{n_t}||c_{n_t}||x_{n_t}^{(s_t)}|\\&\leq C||\mT||\cdot||x||
	\end{align*}
Since such $x$'s form a  dense subspace of  $\ell^1(G) \obp \mA$,  it follows
that $||h_N||\leq C||\mT||$ for all $N$. Now notice that
\[
\psi(h_N)=\sum_{t=1}^N\left|\psi(\mT_{\delta_{h_{n_t}}\otimes
  \rho(e_{n_t})})\right|\geq N\epsilon
\]
for all $N \in \N$. But this is
absurd because $\psi$ is a bounded linear functional and $\{h_N\}$ is a
bounded sequence. Thus, our assumption is wrong, and we must have
        $$
        \lim_{n}\psi(\mT_{\delta_{h_{n}}\otimes a_{n}})=0.
        $$ Thus, $(\varphi_\mU\square \psi)(\mT)=0$ (because the limit
        along an ultrafilter in a compact Hausdorff space is unique and
        every non-principal ultrafilter contains the cofinite
        filter). Since $\psi$ and $\mT$ were arbitrary, we conclude
        that $\varphi_\mU\square \psi=0$ for all $\psi\in X^{**}$.

        Next, we turn to the second Arens product. We have 
        $(\varphi_\mU\Diamond
        \psi)(\mT)=\psi(\mT_{\varphi_{_{\mU}}})$ for all $\psi \in X^{**}$, $\mT \in X^*$. Notice that for any $u\in
        \ell^1(G)\obp \mA$, we have
         \begin{align*}
         \mT_{\varphi_{_{\mU}}}(u)&=\varphi_{\mU}({_u}{\mT})\\&=\lim_{\mU}{_u\mT}(\delta_{h_n}\otimes \rho(e_n))\\&=\lim_{\mU}\mT_{\delta_{h_n}\otimes \rho(e_n)}(u).
        \end{align*}
        Thus,
        $\mT_{\varphi_{_\mU}}=w^*-\lim_{\mU}\mT_{\delta_{h_n}\otimes
          \rho(e_n)}$ and hence
        $\varphi_{_\mU}\Diamond\psi(\mT)=\psi\left(w^*-\lim_{\mU}\mT_{\delta_{h_n}\otimes
          \rho(e_n)}\right)$ for all $\psi \in X^{**}$, $\mT \in X^*$. Further,
        for an arbitrary $u=\sum_{i=1}^\infty \lambda_if_i\otimes
        b_i\in \ell^1(G,\mA)$, we have \begin{align*} \lim_{n}\left|
          \mT_{\delta_{h_{n}}\otimes \rho(e_n)}(u)\right|&= \lim_n
          \left|\mT\left(\sum_{i=1}^\infty\lambda_i\delta_{h_n}\star
          f_i\otimes \rho(e_n)\cdot b_i\right)\right|\\ &=0.
          &&(\because \lim_n\rho(e_n)\cdot b_i=0~\forall i)
      \end{align*} Thus,
        $w^*-\lim_{\mU}\mT_{\delta_{g_n}\otimes
          \rho(e_n)}=0$. Hence, $\varphi_\mU\Diamond\psi(\mT)=0$ for all
        $\psi\in X^{**}$ and all $\mT\in X^*$. 	Thus, we have shown that $\varphi_\mU\square\psi= 0 =\varphi_{\mU}\Diamond \psi$ for all $\psi\in X^{**}$. 
\end{proof}
As a consequence, we can explicitly see the failure of Strong Arens irregularity in $\ell^1(G,\mA)$ when $\mA$ contains a complemented algebraic copy of $\ell^1(\mathbb N)$ (for example Banach algebra $S_1(\mathcal H)$ of trace class operators on a Hilbert space $\mathcal H$).
\begin{corollary}
	$X:=\ell^1(G,\mA)$ is not strongly Arens irregular when $\mA$ contains a complemented algebraic copy of $\ell^1(\mathbb N)$.
\end{corollary}
\begin{remark}
	It has been noticed so far that when $\mathcal A$ is $\mathbb C^n$ with pointwise multiplication or $L^1(H)$ with convolution, then $L^1(G,\mathcal A)$ is strongly Arens irregular for any infinite locally compact group $G$. On the other hand if $\mathcal A$ contains an algebraic copy of  $\ell^1(\mathbb N)$ w.r.t pointwise multiplication, then $\ell^1(G,\mathcal A)$ is not strongly Arens irregular for infinite discrete group $G$. All these examples points toward the following conjecture which is worth exploring.
\end{remark}
\begin{conjecture}
	If $G$ is any infinite locally compact group and $\mathcal A$ is a Banach algebra then $L^1(G,\mathcal A)$ is S.A.I. if and only if $\mathcal A$ is S.A.I..
\end{conjecture}
Recently, it has been proved in \cite{Lav} that if $G$ is a compact abelian group and $\mA$ is a refelexive Banach algebra then $L^1(G,\mA)$ is strongly Arens irregular.\\

{\bf Acknowledgments:} I would like to thank the anonymous referee for  carefully reading the manuscript and providing valuable suggestions and corrections.

\end{document}